\RequirePackage{fix-cm}
\documentclass[smallextended]{svjour3}       \usepackage{graphicx}
\usepackage{mathptmx,geometry}
\usepackage{etoolbox}
\usepackage{xstring}
\usepackage{amsmath}
\usepackage{mathtools}
\usepackage{amssymb}
\usepackage[inline]{enumitem}
\usepackage{comment}
\usepackage{algpseudocode}
\usepackage{tikz}
\usepackage{centernot}
\usepackage{epstopdf}
\usepackage[colorlinks = true, linkcolor = blue, citecolor = blue]{hyperref}

\smartqed  
\renewenvironment{proof}{\begin{origproof}}{\qed\end{origproof}}

\newcommand{\refPre}[2]{\hyperref[#1]{#2\ref{#1}}}
		
		\newcommand{\resType}{\relax}
		\newcommand{\ifResType}[1]{\expandafter\ifstrequal\expandafter{\resType}{#1}}
		\newcommand{\resName}[1]{\StrSplit{#1}{3}{\resType}{\superfluous}\ifResType{cnd} {Condition}
			{\ifResType{sup}
				{Supposition}		
			{\ifResType{thm}
				{Theorem}
			{\ifResType{lem}
				{Lemma}
			{\ifResType{pro}
				{Proposition}
			{\ifResType{cor}
				{Corollary}
			{\ifResType{rem}
				{Remark}
			{\ifResType{ctr}
				{Constraint}
			{\ifResType{exr}
				{Expression}
			{\ifResType{def}
				{Definition}
			{\ifResType{idn}
				{Identity}
			{\ifResType{bnd}
				{Bound}
			{\ifResType{apx}
				{Approximation}
			{\ifResType{pbm}
				{Problem}
			{\ifResType{gam}
				{Game}
			{\ifResType{cha}
				{Chapter}
			{\ifResType{sec}
				{Section}
				{{{}}}}}}}}}}}}}}}}}}}}

\newcommand{\res}[1]{\StrSplit{#1}{3}{\resType}{\superfluous}\resName{\resType}~\ref{#1}}

		\newcommand{\ress}[2]{\StrSplit{#1}{3}{\resType}{\superfluous}\ifResType{cor}{Corollaries}{\makepluralNaive{\resName{\resType}}} \ref{#1} and \ref{#2}}

\newcommand{\numberPre}[1]{\stepcounter{equation}\tag{#1\theequation}}
\newcommand\numberthis{\stepcounter{equation}\tag{\theequation}}

		\makeatletter
		\newcommand{\advanceenumeratelevel}{\advance\@enumdepth\@ne}
		\makeatother
		\let\itemorig\item
		\newcommand{\itemcnd}{\itemorig \stepcounter{cnd}}
		\newenvironment{enumeratecnd}
			{
			\renewcommand{\item}{\itemcnd}
			\begin{enumerate}[label=(C\arabic*)]
			\setcounter{enumi}{\value{cnd}}

			\renewcommand{\theenumii}{\theenumi.\arabic{enumii}.}
			
			\renewcommand{\theenumiii}{\theenumi.\arabic{enumii}.\arabic{enumiii}}
			}
			{
			\end{enumerate}
			}
		\newenvironment{enumeratecnd*}
			{\renewcommand{\item}{\itemcnd}\begin{enumerate*}[label=(C\arabic*)]\setcounter{enumi}{\value{cnd}}\setcounter{enumii}{\value{cnd}}\renewcommand{\theenumii}{\theenumi.\arabic{enumii}.}\renewcommand{\theenumiii}{\theenumi.\arabic{enumii}.\arabic{enumiii}}}{
			\end{enumerate*}
			}
		\newcommand{\beginparacnd}{\begin{inparaenum}[(C1)]
				\setcounter{enumi}{\value{cnd}}
				}
	
		\newcommand{\LHS}{left-hand side}
		\newcommand{\RHS}{right-hand side}
	
	\providecommand{\citet}[1]{\cite{#1}}
	\providecommand{\citep}[1]{(\citet{#1})}

\newcommand{\eg}{for example}

\newcommand{\ie}{that is}

\newcommand{\Aka}{Also known as}

\newcommand{\makepluralNaive}[1]{#1s}
	\newcommand{\makeplural}[1]{\IfEndWith{#1}{y}{ \StrGobbleRight{#1}{1}[\result]\result ies}{#1s}
	}

\newcommand{\propForCor}[1]{The point of \res{pro:#1} is its corollary:}

\DeclarePairedDelimiter{\operatorarg}{(}{)}
\newcommand\DeclareMathOperatorNoLimits[1]{\expandafter\newcommand\csname #1\endcsname[1]{\operatorname{#1}\operatorarg*{##1}}}
\DeclareMathOperatorNoLimits{card}

\newcommand\ProvideMathOperator[2]{\ifdefined#1\else\DeclareMathOperator{#1}{#2}\fi}

\ProvideMathOperator{\Ran}{Ran}
\ProvideMathOperator{\tr}{tr}
\ProvideMathOperator{\proj}{proj}
\ProvideMathOperator{\argmin}{argmin}
\ProvideMathOperator{\cum}{C}
\ProvideMathOperator{\sgn}{sign}
\ProvideMathOperator{\lcm}{LCM}

\newcommand{\limitsZero}{_{i=0}^N}
\newcommand{\limitsOne}{_{i=1}^N}

\newcommand{\setsub}{\backslash}

\newcommand{\Ints}{\mathbb{Z}}
\newcommand{\IntsNonneg}{\Ints^+}
\newcommand{\IntsPos}{\Ints^{>0}}

\newcommand{\lz}[1]{\ell^0\left(#1\right)} \newcommand{\nz}[1]{\left|#1\right|_0}

\newcommand{\placeholderDefn}{More generally, a dot is used as the placeholder for a function argument.}

\newtheorem{thm}{Theorem}
\newtheorem{lem}{Lemma}
\newtheorem{pro}{Proposition}

\newtheorem{defn}{Definition}

\newcounter{cnd}

\newcounter{cnt:rem}

\newcommand{\remStar}{\paragraph{\textbf{Remark}:}}

\begin{document}

\title{Prime-Residue-Class of Uniform Charges on the Integers\thanks{We thank David Donoho for posing the question of whether $PR = R$. The Department of Defense (DoD) supported the first author through the National Defense Science \& Engineering Graduate Fellowship (NDSEG) Program.}}

\titlerunning{Prime-Residue-Class of Uniform Charges}        

\author{Michael Spece       \and
        Joseph B. Kadane }

\institute{M. Spece\at
              Department of Statistics \& Data Science and\\
              Machine Learning Department\\
              Carnegie Mellon University\\
              Pittsburgh, PA 15213\\
              Tel.: 412-268-2717\\
              Fax: 412-268-7828\\
              \email{spece@cmu.edu}           \and
           J.B. Kadane \at
           Department of Statistics \& Data Science \\
              Carnegie Mellon University\\
              Pittsburgh, PA 15213\\
              Tel.: 412-268-8726\\
              Fax: 412-268-7828\\
              \email{kadane@stat.cmu.edu}           }

\date{Received: date / Accepted: date}

\maketitle

\begin{abstract}
There is a probability charge on the power set of the integers that gives probability $1/p$ to every residue class modulo a prime $p$.  There exists such a charge that gives probability $w$ to the set of prime numbers iff $w \in [0,1/2]$. Similarly, there is such a charge that gives probability $x$ to a residue class modulo $c$, where $c$ is composite, iff $x \in [0,1/y]$, where $y$ is the largest prime factor of $c$.\\
\keywords{probability charge \and finite additivity \and uniform distribution \and residue class \and prime numbers}
\subclass{MSC 60A05}
\end{abstract}

\section{Distributions uniform on the integers}
\label{distribution}
A \textbf{probability charge}\footnote{{\Aka} a finitely additive probability or additive distribution, but these terms conflate with terminology for standard probability.}  uniform on the integers assigns $0$ to each integer, but that of $1$ to $\mathbb{Z}$, precluding countable additivity.  Every finite set has probability $0$, and each cofinite set $1$, leaving undetermined jointly infinite-coinfinite sets.

Suppose $\mathcal{C}$ is a collection of subsets from $\Omega$ (here $\mathbb{Z}$) such that $\Omega \in \mathcal{C}$. Let $\mu$ be a non-negative function on $\mathcal{C}$ such that $\mu(\Omega) = 1$.  Theorem 1 of \citet{kad} gives a necessary and sufficient condition that $\mu$ can be extended to a finitely additive probability on the power set of $\Omega$. Applying this result, they show the special case in which $\mathcal{C}$ is the class of sets that have natural\footnote{Also, ``asymptotic" and ``arithmetic".} densities\footnote{Also known as limit (or limiting) relative frequencies.} admits such an extension, where $\mu$ is taken to assign a set its natural density, when that exists.

\citet{sch} study three classes of finitely additive probabilities: The class $L$ extending limit relative frequency, the class $S$ of shift\footnote{Also, ``translation".} invariant\footnote{Equivalently, thinnable with respect to (i) affine transformations; or (ii) $2\times$ or (iii) general-scale invariance.  See Theorem 1.11 of \citet{van}.} functions $\mu$, and the class $R$ assigning probability $1/m$ to each residue class $\bmod$ $m$ for all positive integers $m$. They show that

\begin{equation*}
L \subset S \subset R
\end{equation*}
where each of the inclusions above are strict.

\citet{ker} study the class $WT$ of weakly thinnable probabilities and show it is strictly less inclusive than $L$.

Each of the previously studied classes\footnote{There are others, for example those based on ``thinning out" sets in \citet{van}.}  has an intuitive interpretation of uniformity that goes beyond assigning each integer $0$.  The inclusions indicate that these various notions of uniformity are strictly nested, with $R$ comprising the weakest notion.

\section{The Prime Residue Class}

One may consider a potentially weaker notion of uniformity than that of $R$ by specifying the probability of each residue class $\bmod$ $m$ for $m$ in a strict subset of $\mathbb{Z}$.  The primes are a natural choice for this subset.  Therefore, consider the class $PR$ of finitely additive probabilities on the integers that give probability $1/m$ to each residue class mod $m$ where $m$ is a prime number.  Since this condition is weaker than that defining the class $R$, we must have

\begin{equation*}
WT \subset L \subset S \subset R \subseteq PR
\end{equation*}

Given that every integer has a prime factorization, $R = PR$ is conceivable.  To the contrary, we show that $R \subset PR$.

Since $PR$ is the least demanding of these classes, it is important to establish at the outset that each member of $PR$ is uniform on the integers. That is accomplished with the following easy result.  First, a notation for sets of ``natural" numbers:

\newcommand{\indMem}{i \in \mathcal{\MakeUppercase{i}}}
\newcommand{\whichClassesInds}{$(j_i)_{\indMem} \in \times_{\indMem} \mathcal{J}_i$}
\newcommand{\Primes}{\ensuremath{\mathbb{N}_p}}
\begin{defn}[Natural Numbers]
	Denote the nonnegative integers by $\mathbb{Z}^+$ and the positive ones by $\mathbb{Z}^{>0}$.
	
	Also, let {\Primes} denote the primes, $\Ints$ integers, $\mathbb{Q}$ rationals, and $\mathbb{R}$ reals.
\end{defn}

\begin{pro}[Uniformity]
	\label{new pro}
	Under $PR$, each integer has probability $0$.
\end{pro}
\begin{proof}
	Let $z \in \Ints$ and $\epsilon > 0$ be given, and let $m$ be a prime greater than $1/\epsilon$. Now $z \in j \bmod m$ for some $j, 0 \leq j \leq m-1$. Then
	\[P\{z\} \leq P\{j \bmod m\} = 1/m < \epsilon.
	\]
	Hence,  $P\{z\} = 0$ for all $z \in \Ints$.
\end{proof}

The remainder of this paper is organized as follows: \res{sec:suprema-of-probabilities} gives general upper bounds on the probability of sets. \ress{sec:application-to-primes}{sec:app:res} respectively apply these results to $\mathbb{N}_p$ and residue classes, proving the claims thereon.   \res{sec:conclusion} concludes.

\section{Suprema of Probabilities}
\label{sec:suprema-of-probabilities}
\newcommand{\prob}{probability of the prime}
For every class of uniform distributions and every subset $S$ of the space measured under these distributions, the probability range of $S$ is a closed interval (Theorem 2 of \citet{kad}).  Restricting to the class $PR$ and the set $\mathbb{N}_p$ of prime numbers, a greatest lower bound of $0$, quoted below for future reference, is immediate from the inclusion $L \subset PR$.
\begin{pro}[Greatest Lower Bound for Primes]
	\label{pro:lower:primes}
\[	\inf_{\mu \in PR} \mu(\mathbb{N}_p) = 0. 
\]
\end{pro}
Thus, to prove the {\prob} numbers can be any value in the interval $[0, 1/2]$, it suffices to show the least upper bound of the \prob s is $1/2$.

At the heart of affording measure to a set is the following theorem about general sets.  

\newcommand{\domInit}{Suppose $\mathcal{C}$ is a collection of subsets of $2^\Omega$}
\newcommand{\constraintsKad}{h\in \mathbb{Z}^{>0};a,b \in \mathbb{Z}^+; A_i, B_j \in \mathcal{C}}
\newcommand{\refConds}[2]{Conditions \refPre{#1}{C} and \refPre{#2}{C}}
\newcommand{\whichClasses}{For \whichClassesInds, $\cap_{\indMem} j_i \bmod p_i $ contains infinitely many }
\newcommand{\propIsctShft}{\frac{\card{\overline{\mathcal{J}}_i}}{p_i}}

\newcommand{\rel}{\geq}\newcommand{\relPrime}{\leq}\newcommand{\ext}{\sup}\newcommand{\extPrime}{\inf}\newcommand{\probset}{S}
\newcommand{\mSpace}{\Omega}
\newcommand{\mBasis}{\mathcal{C} \subseteq \sigalg}
\begin{thm}[Probability Range]
\label{thm:kad}
Suppose $\mathcal{C}$ is a subset of the power-set $2^\mSpace$, that $\Omega \in \mathcal{C}$, and finally that $\mu_0$ is a function on $\mathcal{C}$ that can be extended to a finitely additive probability on $2^\Omega$.  Let $\mathcal{M}$ be the family of such extensions.  Then for every $\probset \in 2^\Omega$,
\[\ext_{\mu \in \mathcal{M}} \mu(\probset) = \extPrime \left\{\frac{1}{h}\left(\sum_{i=1}^a \mu_0(A_i) - \sum_{j=1}^b \mu_0(B_j)\right)\right\}
\numberPre{Id}\label{idn:kad}
,\]
in which the $\extPrime$ is taken over $\constraintsKad$ such that
\[
\sum_{i=1}^a I_{A_i} - \sum_{j=1}^b I_{B_j} \rel h I_\probset \numberPre{Ct} \label{ctr:kad}
.\]
\end{thm}
\begin{proof}
The proof is that of Theorem 2 in \citet{kad}.
\end{proof}

\bigskip
\bigskip
\noindent
\textit{\textbf{Remark}} (Notation):
\begin{itemize}
	\item $\probset$ typically denotes a set whose probability is of interest.
	\item The prefix of a number is an abbreviation.  For example, \ref{ctr:kad} is short for Constraint \arabic{equation}, and is referred to as \res{ctr:kad} in prose.  ``A" stands for approximation, the meaning of which is to be interpreted liberally.
\end{itemize}

\newcommand{\binop}{+}\newcommand{\divd}{\textnormal{ divides }}
\newcommand{\ndiv}{\textnormal{ does not divide }} \begin{defn}[Relations]
For all sets $A$ and $B$, define $A \setsub B:= \{a \in A : a \notin B\}$ and, when $+$ is defined on $A \times B$, $A \binop B := \{ a \binop b :a \in A, b \in B\}$.

The elements of a residue class are \textbf{shifts} or, less descriptively, \textbf{representatives}.  For all $A \subseteq \mathbb{Z}$ and $m \in \IntsPos$, $A \bmod m := \cup_{a \in A} a \bmod m$.

For all $a,b \in \mathbb{Z}^{>0}$, $a | b$ denotes that $a$ divides $b$.  For every pair of $\mathbb{R}$-valued functions $f$ and $g$ with common domain $X$, $f \leq g$ is to be understood in the domination or Pareto sense, {\ie} for all $x \in X$, $f(x) \leq g(x)$.
\end{defn}

\begin{defn}[Classes of Functions]
For a function $f$ with co-domain $Y \subseteq \mathbb{Q}$, let $\nz{f}$ be the counting measure of $f^{-1}(Y \setsub \{0\})$, that is the number of elements in the domain of $f$ that are not mapped to zero; and for a collection of functions $\mathcal{F}$, $\lz{\mathcal{F}}:=\{f\in\mathcal{F}:\nz{f} \ne \infty\}$.  We often specify a collection of functions merely by their domain and co-domain, as follows:

For all sets $A,B$, let $A^B$ denote the collection of functions mapping from $B$ into $A$.
\end{defn}

The difference of sums, in the form that appear in \res{idn:kad} and \res{ctr:kad}, can be re-written as a single sum, as follows.

\begin{lem}[Concise Rewrite]
\label{lem:canon}
Let $\mathcal{C}$ be a collection of sets and $\mathcal{F} := \lz{\mathbb{Q}^{\mathcal{C}}}$.  For all $\constraintsKad$, there exists a unique $f \in \mathcal{F}$ such that, for all $g \in \mathbb{R}^\mathcal{C}$,
\[
\frac{1}{h}\left(\sum_{i=1}^a g(A_i) - \sum_{j=1}^b g(B_j)\right)
 =
\sum_{C \in \mathcal{C}} f(C) g(C)
\numberPre{Id}\label{idn:canon}
.\]
Conversely, if $f \in \mathcal{F}$ and $g \in \mathbb{R}^\mathcal{C}$, then there exists $\constraintsKad$ such that \res{idn:canon} holds.
\end{lem}
\begin{proof}
Suppose $\constraintsKad$.

Uniqueness follows from the flexibility in choosing $g$: for all $C \in \mathcal{C}$,
\begin{align*}
f(C) = \frac{1}{h}\left(\sum_{i=1}^a I_{A_i = C} - \sum_{j=1}^b I_{B_j = C}\right)
\numberPre{Id}\label{idn:coefficient}
.\end{align*}

The question becomes whether, for all $g$,
\begin{align*}
\sum_{C \in \mathcal{C}} \left(\sum_{i=1}^a I_{A_i = C} - \sum_{j=1}^b I_{B_j = C}\right) g(C)
&=
\left(\sum_{i=1}^a g(A_i) - \sum_{j=1}^b g(B_j)\right)
,\end{align*}
for which the following system (if true) suffices
\begin{align*}
\sum_{C \in \mathcal{C}} \sum_{i=1}^a I_{A_i = C} g(C) = \sum_{i=1}^a g(A_i)
\\
\sum_{C \in \mathcal{C}} \sum_{j=1}^b I_{B_j = C}g(C)
=\sum_{j=1}^b g(B_j)
.\end{align*}
But indeed
\begin{align*}
\sum_{i=1}^a g(A_i) = \sum_{i=1}^a \sum_{C \in \mathcal{C}} I_{A_i = C}g(C) = \sum_{C \in \mathcal{C}}\sum_{i=1}^a I_{A_i = C}g(C)
\end{align*}
and symmetrically for $\sum_{j=1}^b g(B_j)$.

Turning to the converse, invert \res{idn:coefficient}.  Since therein only $A_i$ and $B_j$ that equal $C$ count, $\card{i:A_i = C}$ and $\card{j:B_j = C}$ can be set without regard to $a$, $b$, $(A_i)_{i:A_i \ne C}$, or $(B_j)_{j:A_j \ne C}$.  Set $h$ to a common denominator of $\{f(C):C \in \mathcal{C}\}$ so that $hf(C)$ is always an integer.  Set $\card{i:A_i = C} := hf(C)I_{f(C) > 0}$ and $\card{j:B_j = C} := hf(C)I_{f(C) < 0}$.
\end{proof}

Rewriting the sums appearing in \res{thm:kad} according to \res{lem:canon}, one obtains the following theorem restatement.
\begin{thm}[Concise Form of \res{thm:kad}]
\label{thm:range:lz}
Suppose $\mathcal{C} \subseteq 2^\mSpace$, that $\Omega \in \mathcal{C}$, and finally that $\mu_0$ is a function on $\mathcal{C}$ that can be extended to a finitely additive probability on $2^\Omega$.  Let $\mathcal{M}$ be the family of such extensions.   Then for every $\probset \in 2^\Omega$, \[\ext_{\mu \in \mathcal{M}} \mu(\probset) = \extPrime_{\alpha \in \lz{\mathbb{Q}^\mathcal{C}}} \sum_{C \in \mathcal{C}} \alpha(C) \mu_0(C)
\numberPre{Id}\label{idn:optimization:canon}
,\]
in which $\alpha$ is additionally subject to 
\[
\sum_{C \in \mathcal{C}} \alpha(C) I_C \rel I_\probset \numberPre{Ct} \label{ctr:s}
.\]
\end{thm}
\begin{proof}
Applying \res{lem:canon} simultaneously to $\mu_0,I \in \mathbb{R}^\mathcal{C}$ in  \res{idn:kad} and \res{ctr:kad}, respectively, yields the desired form of the objective function in the {\RHS} of \res{idn:optimization:canon} and in \res{ctr:s}.  The converse in \res{lem:canon} ensures the constraint set has not expanded.

\end{proof}

\remStar  In contrast to \res{thm:kad}, each $C \in \mathcal{C}$ appears in at most one term in each of the sums appearing on the right (resp. left) hand side of \res{idn:optimization:canon} and \res{ctr:s}; therein, also, $I_S$ has no coefficient.  Both of these simplifications, especially the first, make bookkeeping easier when applying the theorem to $PR$.

\newcommand{\resPos}[1]{k \in j#1 \bmod p_i \cap \IntsPos}
\begin{defn}[Subsets of Primes]
	\label{def:primes}
	Let $p_i$ be the $i$th largest prime ($i \in \IntsPos$), $p_0 := 1$, and $\{1, \cdots, p_i\}$ or $\{1 : p_i\}$ the set of positive integers less than $p_i + 1$.
	
	For all $N \in \IntsNonneg$, let $N!_p := \prod\limitsOne p_i$ and, for all $j \in \times\limitsOne \{1, \cdots, p_i\}$, $s(j)$ the shift of $\cap\limitsOne j_i \bmod p_i$ in $\{1, \cdots, N!_p\}$ (which exists and is unique by the Chinese Remainder Theorem).  Note $s$ depends on $N$ implicitly.
\end{defn}

\begin{defn}[$\cdot$]
	Let $\ell_{\mathbb{Q}} := \lz{ \mathbb{Q}^{\IntsNonneg \times \IntsPos} } $.
For all $(\alpha_{i,j})_{i,j} \in \ell_\mathbb{Q}$, let $\alpha_{i,\cdot}$ denote the vector $\alpha_{i,1}, \alpha_{i,2}, \cdots$.  \placeholderDefn \end{defn}

\res{thm:range:lz}---with $\mathcal{M} := PR$, $\mathcal{C}$ the collection of prime residue classes, and $\mu_0(\cdot \bmod p_i) \equiv \frac{1}{p_i}$ for all $i \in \IntsNonneg$---gives
\begin{align*} 
\ext_{\mu \in PR} \mu(\probset) = \extPrime_{\alpha \in \ell_\mathbb{Q}} f(\alpha)
\numberPre{Id}\label{idn:objective:start}
,\end{align*}
in which $f(\alpha) := \sum_{i=0}^\infty \sum_{j=1}^{p_i}  \frac{\alpha_{i,j}}{p_i}$ and $\alpha$ is additionally subject to 
\[
G(\alpha) \rel I_\probset
\numberPre{Ct}\label{ctr:start}
,\]
where $G(\alpha) := \sum_{i=0}^\infty \sum_{j=1}^{p_i} \alpha_{i,j}I_{j \bmod p_i}$.

The following proposition provides a convenient avenue through which to make explicit the dependence of $\alpha \in \ell_\mathbb{Q}$ (as appears in \res{idn:objective:start}) on only a finite number of components.

\begin{pro}[Continuity]
	\label{pro:continuity}
	Suppose \begin{enumerate*}[label=(S\arabic*)] \item $f:A \to \mathbb{R}$ and \item \label{sup} $\cup_{n = 1}^\infty A_n = A$.  \end{enumerate*}
	Then
	\[
	\lim_{N \to \infty} \extPrime_{\alpha \in \cup_{n=1}^N A_n} f(\alpha) =
	\extPrime_{\alpha \in A} f(\alpha)
	\numberPre{Id}\label{idn:supremum:limit}
	.\]
\end{pro}
\begin{proof}
	$\extPrime_{\alpha \in \cup_{n=1}^N A_n} f(\alpha)$ over the weakening constraint $\alpha \in \cup_{n=1}^N A_n$ is monotonic over $N$.  Hence, the {\LHS} of \res{idn:supremum:limit} is well defined.
	
	Because, for all $N \in \IntsPos$, $\cup_{n=1}^N A_n \subseteq A$;
	\[
	\lim_{N \to \infty} \extPrime_{\alpha \in \cup_{n=1}^N A_n} f(\alpha)
	\rel
	\extPrime_{\alpha \in A} f(\alpha)
	\numberPre{A}\label{apx:weaker}
	,\] so it suffices to prove the reverse inequality of \res{apx:weaker}.
	
	By completeness of the real numbers, there exists $\left(\alpha_n\right)_{n=1}^\infty \in A^\infty$ such that \[\lim_{n \to \infty} f(\alpha_n) = \extPrime_{\alpha \in A} f(\alpha)
	\numberPre{Id}\label{idn:limit}
	.\]
	By \res{sup}, there exists a subsequence $(N_k)_{k=1}^\infty$ of $(N)_{N=1}^\infty$ such that, for all $k \in \IntsPos$, $\alpha_k \in \cup_{n=1}^{N_k} A_n$.  Then
	
	\begin{align*}
	\lim_{N \to \infty} \extPrime_{\alpha \in \cup_{n=1}^N A_n} f(\alpha)
	&=\hspace{.05em}
	\lim_{k \to \infty} \extPrime_{\alpha \in \cup_{n=1}^{N_k} A_n} f(\alpha)
	\numberPre{Id}\label{idn:subsequence}
	.\end{align*}
	Because, for all $k \in \IntsPos$, $\extPrime_{\alpha \in \cup_{n=1}^{N_k} A_n} f(\alpha) \relPrime f(\alpha_k)$;
	the {\RHS} of \res{idn:subsequence} is upper bounded by
	\[
	\lim_{k \to \infty} f(\alpha_k) \numberPre{A}\label{apx:limit}
	.\]
	
	Combining \res{idn:limit} and \res{apx:limit}, the reverse of \res{apx:weaker} holds.
\end{proof}

The theorem following relies on the notion of multi-sets.

\begin{defn}[Multi-sets]
	\label{def:multiset}
	\textbf{Multi-set} is an extension of set, endowing each of its \textbf{elements} with a multiplicity in $\IntsPos \cup \{\infty\}$.  Every set can be viewed as a multi-set each of whose elements has multiplicity 1.  $\mathcal{J}$, or an embellishment thereof, denotes a multi-set.

	The \textbf{cardinality} of a multi-set is the sum of its elements' multiplicities.  Multi-set and set cardinalities coincide on countable sets.
	
Multi-set intersection also extends that for set.  \textbf{Intersections} of multi-sets are taken to preserve the highest multiplicity of an element appearing in any of the multi-sets (excluding a multiplicity of zero).
	
	Multi-set \textbf{inclusion} similarly takes multiplicity into consideration: $\mathcal{J}_0 \subseteq \mathcal{J}$ iff $\mathcal{J}_0 \cap \mathcal{J}$ does not increase the multiplicity of any element.  However, multi-set inclusion in a set is only meant to indicate the elements of that multi-set are in that set.
\end{defn}

\newcommand{\inds}{\mathcal{N}}
\newcommand{\indsDef}{\{1, \cdots, N\}}
\begin{defn}[Coordinates]
\label{def:coordinates}

	Given, either implicitly by its use or explicitly by its declaration, an $N \in \IntsPos \cup \{\infty\}$, let $\inds := \begin{cases} \indsDef & N < \infty \\ \IntsPos & N = \infty\end{cases}$.
	
	For all $n \in \inds$,  finite-cardinality multi-sets $A \subseteq \times_{n \in \inds} \{1, \cdots, p_n\}$, $a \in A$, $i \in \IntsPos$, and $b \in \{1, \cdots, p_i\}$,
	\begin{itemize} 
		\item $a_n$ is the $n$ (of $N$)th coordinate of $a$
		\item $\proj_n (A)$ (projection) is the set of $n$th coordinates from $A$, endowing each coordinate value with a multiplicity equal to the cardinality of $\proj_n^{-1}(a) \cap A$,\footnote{That $A$ is finite-cardinality ensures the endowed multiplicities are finite.} in which the inverse image is taken with respect to regarding $\times_{n \in \inds} \{1, \cdots, p_n\}$ as the domain of $\proj_n$.
		\item $\textnormal{card}_n(A) := \card{\proj_n (A)}$
		\item $a_{-i} := \left(a_1, \cdots, a_{i-1}, a_{i+1}, \cdots, a_N\right)$
		\item $\left(a_{-i}, b\right) := \left(a_1, \cdots, a_{i-1}, b, a_{i+1}, \cdots, a_N\right)$.
	\end{itemize}
	
	Similarly, for all sets $A, B \in \times_{n=1}^N \{1, \cdots, p_n\}$ and $i \in \{1, \cdots, N\}$, $A_{-i} \times B_i := \{(a_{-i}, b_i) : a \in A, b \in B\}$.
\end{defn}

\begin{thm}[Probability Range over $PR$]\label{thm:range:prime}
	For every $\probset \subseteq \Ints$,\[\ext_{\mu \in PR} \mu(\probset) \rel 
	\lim_{N \to \infty}
	\max_{\mathcal{J}}\frac{
		\card{
			\mathcal{J}
			\cap s^{-1}(S \bmod N!_p)}}
		{N!_p}
	\numberPre{A}\label{apx:optimization:primes}
	,\]
	in which 
	\begin{enumeratecnd}
		\item \label{cnd:multiset}
		For all $N \in \IntsNonneg$, the $\mathcal{J}$ of
		\[
		\max_{\mathcal{J}}\frac{
			\card{
				\mathcal{J}
				\cap s^{-1}(S \bmod N!_p)}}
		{N!_p}
		\] within the {\RHS} of \res{apx:optimization:primes} varies over multi-sets such that (i) $\mathcal{J} \subseteq \times\limitsOne \{1, \cdots, p_i\}
		$ and (ii) for all $n \in \inds$, $\proj_n(\mathcal{J}) = \{1^{N!_p/p_n}, \cdots, p_n^{N!_p/p_n}\}$; in which superscripts denote the multiplicity of elements in the multi-set with non-positive subscripts signifying the absence of that element.	
\end{enumeratecnd}
\end{thm}
\newcommand{\splitInds}{(j_0,j) \in \times\limitsZero \{1, \cdots, p_i\}}
\newcommand{\stateThinDef}{ \left[\times_{i \in \IntsPos} \mathcal{H}_i\right] \cup \left[\cup_{i_* \in \tilde{\mathcal{K}}} ([\times_{i \in \IntsPos \setsub \{i_*\}} \mathcal{H}_i] \times \mathcal{K}_{i_*})\right]
}
\begin{proof}

	Invoking \res{pro:continuity}, with $A_n := \{\alpha \in \ell_\mathbb{Q} : \nz{\alpha} \leq n\}$, the right-hand side of \res{idn:objective:start}
	\[
	\sup_{\mu \in PR}\mu(S)
	\numberPre{P}\label{pbm:extreme:PR}
	\]
	 becomes
	\begin{align*}
		\lim_{N \to \infty} \extPrime_{\alpha \in \ell_\mathbb{Q} : \nz{\alpha} \leq N} \sum_{i=0}^\infty \sum_{j=1}^{p_i} \frac{\alpha_{i,j}}{p_i}
		=
	\lim_{N \to \infty} \extPrime_{\alpha \in \ell_\mathbb{Q}} \sum_{i=0}^N \sum_{j=1}^{p_i} \frac{\alpha_{i,j}}{p_i}
	\numberPre{Id}\label{idn:finite}
	,\end{align*}
	in which, for all $N \in \IntsNonneg$, the $\alpha$ of 
	\[
	\extPrime_{\alpha \in \ell_\mathbb{Q}} \sum_{i=0}^N \sum_{j=1}^{p_i} \frac{\alpha_{i,j}}{p_i}
	\numberPre{}\label{exr:finite}
	\]
	within the {\RHS} of \res{idn:finite} is additionally subject to 
	\[
	\sum_{i=0}^N \sum_{j=1}^{p_i} \alpha_{i,j}I_{j \bmod p_i} \rel I_\probset
	\numberPre{Ct} \label{ctr:dominate}
	.\]
	By the Chinese Remainder Theorem, every $z \in \Ints$ corresponds to a \[j_{z, \cdot} \in \times\limitsZero \{1, \cdots, p_i\} \] such that $z \bmod N!_p = s\left( j_{z,-0} \right) \bmod N!_p$.  ($j$ as defined is a branch of $s^{-1}$.)  Moreover, $z$ is in no other residue classes than $s(j_{z,k}) \bmod p_k$ for all positive integers $k \leq N$.  Therefore, the left-hand side of \res{ctr:dominate} evaluated at $z$ is $\sum_{i=0}^N \alpha_{i,j_{z,i}}$ and \res{ctr:dominate} can be re-written as 
	\[
	 \sum_{(j_0,j) \in \times\limitsZero \{1, \cdots, p_i\}}
	 \left(\sum_{i=0}^N \alpha_{i, j_i}\right)
	 I_{s(j)\bmod N!_p} 
	 \rel I_S
	 \numberPre{Ct} \label{ctr:dominate:product}
	.\]
	
	\res{ctr:dominate:product} can be decomposed point-by-point as the set of constraints
	\begin{align*}
	&\left\{\sum_{(j_0, j) \in \times\limitsZero \{1, \cdots, p_i\}}  I_{s(j) \bmod N!_p}(z) \sum_{i=0}^N \alpha_{i, j_i}
	\rel I_\probset(z)\right\}_{z \in \Ints}
	\numberPre{Ct} \label{ctr:points}
	.\end{align*}
	Breaking the left and right-hand sides of the inequalities \ref{ctr:points} into cases of which modular class representative and which value in the indicator range are picked out by $z$, respectively:
	\begin{align*}
	&\left\{
	\begin{array}{@{}lr@{}}
	\sum_{i=0}^N \alpha_{i, j_i} & z \in s(j) \bmod N!_p
	\end{array}	
	\right\}_{\splitInds}
	\rel
	\begin{cases}
	0 & z \in \Ints \setsub S
	\\
	1 & z \in S
	\end{cases}
	\\=&
	\left\{
	\sum_{i=0}^N \alpha_{i, j_i}
	\rel
	\begin{cases}
	0 & z \in \Ints \setsub S \cap s(j) \bmod N!_p
	\\
	1 & z \in S \cap s(j) \bmod N!_p
	\end{cases}
	\right\}_{\splitInds}
	&\textnormal{Combine conditions}
	\\=&
\left\{
	\sum_{i=0}^N \alpha_{i, j_i} \rel
	\begin{cases}
	 1 & S \cap s(j) \bmod N!_p \ne \emptyset
	\\
	 0
	 & \textnormal{otherwise}
	\end{cases}
	\right\}_{\splitInds}
	&\textnormal{Combine inequalities}
	\\=&
	\sum_{i=0}^N \alpha_{i, j_i} \rel
	\begin{cases}
	1 & (j_0,j):s(j) \in S\bmod N!_p
	\\
	0 & \textnormal{otherwise}
	\end{cases}
	&\textnormal{Condition on $j$}
	\\=&
	\sum_{i=0}^N \alpha_{i, j_i} \rel
	\begin{cases}
	1 & j \in \{1\} \times s^{-1}(S \bmod N!_p)
	\\
	0 & \textnormal{otherwise}
	\end{cases}
	& \textnormal{Definition of $\cdot^{-1}$}
	\numberPre{Ct} \label{ctr:dominate:binary}
	.\end{align*}
	
	\res{exr:finite} is, dividing out repetitions,
	\[
	\extPrime_{\alpha \in \ell_\mathbb{Q}} \sum_{j \in \times\limitsZero \{1, \cdots, p_i\}} \sum_{i=0}^N  \frac{1}{\prod_{k \ne i} p_k} \frac{\alpha_{i,j}}{p_i}
	\numberthis\label{exr:repeat}
	,\]
	due to there being $\prod_{k \ne i} p_k$ repetitions of each $\{1, \cdots, p_i\}$ in the $i$th component of \[\times\limitsZero \{1, \cdots, p_i\}
	.\]
	Factoring out (the constant) $\frac{1}{N!_p}$ from \res{exr:repeat} obtains
	\[
	\extPrime_{\alpha \in \times\limitsZero \mathbb{Q}^{p_i}}  \frac{1}{N!_p} \sum_{j \in \times \limitsZero \{1, \cdots, p_i\}} 
	\sum_{i=0}^N \alpha_{i, j_i}
	\numberPre{}\label{exr:sums:alpha}
	.\]

The indices of $\alpha$ being summed over in  \[
\sum_{j \in \times\limitsZero \{1, \cdots, p_i\}} 
\sum\limitsZero \alpha_{i, j_i}
\numberPre{}\label{exr:sums:alpha:set}
\] of \res{exr:sums:alpha} forms a multi-set of pairs \[
\left\{(i, j)^{\begin{cases}0 & j > p_i \\ N!_p/p_i & j \in \{1, \cdots, p_i\}\end{cases}}
\right
\}
\numberPre{}\label{exr:indices:multiset}
\]  
whose multiplicities depend on the second coordinate $j$.  (Because $p_i$ divides $N!_p$, the multiplicities are non-negative integers, and thus valid.)

For every $\mathcal{J}$ obeying Condition \ref{cnd:multiset}, the indices of $\alpha$ being summed over in 
\[
\sum_{j \in \{1\} \times \mathcal{J}} 
\sum\limitsZero \alpha_{i, j_i}
\numberPre{}\label{exr:sums:alpha:multiset}
\] forms the multi-set \ref{exr:indices:multiset} as well.  Therefore, sums \ref{exr:sums:alpha:set} and \ref{exr:sums:alpha:multiset} must have the same multi-sets of terms and so, themselves being the sums of precisely the terms thereof, must be equal:
\[
\sum_{j \in \times\limitsZero \{1, \cdots, p_i\}} 
\sum\limitsZero \alpha_{i, j_i}
=
\sum_{j \in \{1\} \times \mathcal{J}} 
\sum\limitsZero \alpha_{i, j_i}
\numberPre{Id}\label{idn:sums:alpha}\]

By \res{idn:sums:alpha}, \res{exr:sums:alpha}, and in turn Expressions \ref{exr:repeat} and \ref{exr:finite}, are equal to 
\[
\extPrime_{\alpha \in \times\limitsZero \mathbb{Q}^{p_i}}  \frac{1}{N!_p} \sum_{j \in \{1\} \times \mathcal{J}}
\sum_{i=0}^N \alpha_{i, j_i}
\numberPre{}\label{exr:sum:aplha}
,\]
for every $\mathcal{J}$ obeying \ref{cnd:multiset}.

Sums \ref{exr:sums:alpha:multiset} can be split up as
\[
\sum_{j \in \{1\} \times \left[\mathcal{J} \cap s^{-1}(S \bmod N!_p)\right] } 
\sum\limitsZero \alpha_{i, j_i} + \sum_{j \in \{1\} \times \left[\mathcal{J} \setsub s^{-1}(S \bmod N!_p) \right]} 
\sum\limitsZero\alpha_{i, j_i}
\numberPre{}\label{exr:rearranged}
.\] 

Imposing \res{ctr:dominate:binary} on \res{exr:rearranged} yields a lower bound, that of
\begin{align*}
&\sum_{j \in \{1\} \times \left[\mathcal{J} \cap s^{-1}(S \bmod N!_p)\right] } 
1 + \sum_{j \in \{1\} \times \left[\mathcal{J} \setsub s^{-1}(S \bmod N!_p) \right]} 
0
 \\=& \card{\mathcal{J} \cap s^{-1}(S \bmod N!_p)}
\numberPre{A}\label{apx:sum:binary}
,\end{align*}
in which the equality leading to \ref{apx:sum:binary} is by the definition of cardinality (\ref{def:multiset}).

Combining \res{exr:rearranged} and \res{apx:sum:binary}, \res{exr:finite} is bounded below by
\begin{align*}
&\max_{\mathcal{J}} \extPrime_{\alpha \in \times_{i=1}^N \mathbb{Q}^{p_i}}  \frac{\card{\mathcal{J} \cap s^{-1}(S \bmod N!_p)}}{N!_p}
=
\max_{\mathcal{J}}  \frac{\card{\mathcal{J} \cap s^{-1}(S \bmod N!_p)}}{N!_p}
\numberPre{A}\label{apx:sums:cardinality}
,\end{align*} in which, for all $N \in \IntsPos$, $\mathcal{J}$ obeys \res{cnd:multiset}.  (The {\LHS} of \res{apx:sums:cardinality} does not depend on $\alpha$, so can be dropped to obtain the {\RHS}.)

Taking $N \to \infty$ in \res{apx:sums:cardinality} and \res{exr:finite}, while equating the latter limit to the {\LHS} of \res{idn:objective:start}, yields \res{apx:optimization:primes}.
\end{proof}

\remStar The $\mathcal{J}$ of \res{cnd:multiset} is in some sense a re-arrangement of the components of $\times\limitsOne \{1, \cdots, p_i\}$.
\remStar The requirement $\proj_n(\mathcal{J}) = \{1^{N!_p/p_n}, \cdots, p_n^{N!_p/p_n}\}$ of \res{cnd:multiset} can be weakened to \[\proj_n(\mathcal{J}) \subseteq \{1^{N!_p/p_n}, \cdots, p_n^{N!_p/p_n}\} \label{}. \]  Moreover, the maximum in \res{apx:optimization:primes} is attained by some $\mathcal{J}$ for which $\proj_n(\mathcal{J}) \subseteq \proj_n(\mathcal{J}) \cap \proj_n\left(s^{-1}(S \bmod N!_p)\right)$ (where the intersection obeys \res{def:multiset}).

The following proposition further concerns the projections of a multi-set (which appear in the constraint set specified in \res{cnd:multiset}).
\begin{pro}[Exchange Preserving $\proj_n(\mathcal{J})$]
	\label{pro:projection}
Let $N \in \IntsPos \cup \{\infty\}$, $\mathcal{J} \subseteq \times_{n \in \inds} \{1, \cdots, p_n\}$, a multi-set with at least two elements, say $j,k$, and $n_* \in \inds$.  Let $\mathcal{J}':= \mathcal{J} \setsub \{j,k\} \cup \{j', k'\}$; where \begin{enumerate*} \item $j' := (j_{-{n_*}}, k_{n_*})$, \item $k' := (k_{-n_{*}}, j_{n_*})$, \item $\setsub$ decrements multiplicities as follows: for all objects $a$ and $b,c \in \Ints$, $\{a^b, \cdots\}\setsub \{a^c, \cdots\} = \{a^{b-c}, \cdots\}$, and \item $\cup$ increments them as follows: for all objects $a$ and $b,c \in \IntsNonneg$, $\{a^b, \cdots\} \cup \{a^c, \cdots\} = \{a^{b+c}, \cdots\}$.  \end{enumerate*}  Then, for all $n \in \inds$, $\proj_n(\mathcal{J}') = \proj_n(\mathcal{J})$.
\end{pro}
\begin{proof}
$j,k \in \mathcal{J}$ implies subtracting (or adding) $\{j,k\}$ has the reverse effect on $\proj_n$ of the resultant multi-set, for all $n \in \inds$, as adding $\{j',k'\}$ (or subtracting, respectively). 
\end{proof}

Given the frequency with which $\times_{n \in \inds} \{1, \cdots, p_n\}$ occurs, a descriptive notation specific to its elements is given.
\newcommand{\passPlural}{passing through, in their $i$th components, }
\begin{defn}[Paths]
	Given any $N \in \IntsPos \cup \{\infty\}$, multi-set $\mathcal{J} \subseteq \times_{n \in \inds} \{1, \cdots, p_n\}$, $m \in \IntsPos$, and \[
	b^m \in \times_{n \in \inds} \{1, \cdots, p_n\}
	,\] $b$ is a \textbf{path}.  An illustration of the path $(1, 2)$ ($N = 2$):
	\begin{tikzpicture}[
	  font=\sffamily,
	  every matrix/.style={ampersand replacement=\&,column sep=.5cm,row sep=0cm},
	  source/.style={draw,thick,rounded corners,fill=yellow!20,inner sep=.3cm},
	  process/.style={draw,thick,circle,fill=yellow!20},
	  sink/.style={source,fill=green!20},
	  dots/.style={gray,scale=2},
	  to/.style={->,shorten >=1pt,semithick,font=\sffamily\footnotesize},
	  every node/.style={align=center}]
	
\matrix{
	      			\& \node[process] (daq) {2}; 
	      			\\
	  \node[process] (1) {1};\& \node (buffer) {1};
	    \\
	  };
	
\draw[] (1) -- node[midway,right] {} (daq);
	\end{tikzpicture}, where each column array of integers corresponds to a set in the product $\times_{n \in \inds} \{1, \cdots, p_n\}$.
	
	For every $j \in \IntsPos$ and set $A \subseteq \{1, \cdots, p_j\}$, $b$ \textbf{passes} through, in its $j$th component, $A$ iff $b_j \in A$.  In the illustration above, the path shown passes through $\{1\}$ (both $\{2\}$ and $\{1,2\}$) in the first (second) component, respectively.
\end{defn}

Combinatorial arguments are described using the following terminology.
\begin{defn}[Life]
	Given any $N \in \IntsPos \cup \{\infty\}$, set $A \subseteq \times_{n \in \inds} \{1, \cdots, p_n\}$, and multi-set $\mathcal{J} \subseteq\times_{n \in \inds} \{1, \cdots, p_n\} $, a path in $\mathcal{J}$ is \textbf{alive} iff it is in $A$.  Exchanging components as in \res{pro:projection} is \textbf{enlivening} iff $j$ is dead and $j'$ alive or $k$ is dead and $k'$ is alive; and is (life) \textbf{preserving} iff, for each $i \in \{j,k\}$ that is alive, so is $i'$.  An enlivening and preserving exchange between an alive path and dead one is a (successful) \textbf{donation} (that leaves both $j',k'$ alive).
\end{defn}

\begin{pro}[Re-Directing Paths]
	\label{pro:redirect}
	Fix arbitrary  \begin{enumerate} \item $i, j \in \IntsPos$ (not necessarily distinct), \item a multi-set $\mathcal{J} \subseteq \times_{n \in \IntsPos} \{p_1, \cdots, p_n\}$ of finite cardinality, \item and, $\forall k \in \{i,j\}$, $A_k \subseteq \{1, \cdots, p_k\}$. \end{enumerate}  Suppose \begin{enumeratecnd*} \item \label{cnd:paths} $\textnormal{card}\left({\proj_i(\mathcal{J}) \cap A_i}\right) \leq \textnormal{card}\left({\proj_j(\mathcal{J}) \cap A_j}\right)$.\end{enumeratecnd*}  Then there exists a multiset $\mathcal{J}'$ such that \begin{enumeratecnd*} \item\label{cnd:preserve} $\proj_n(\mathcal{J}') = \proj_n(\mathcal{J}) $ for all $n \in \inds$ and \item \label{cnd:subsume} for all $v \in \mathcal{J}'$, $v_i \in A_i \implies v_j \in A_j$.\end{enumeratecnd*}
\end{pro}
\begin{proof}
	Suppose $i = j$.  Substituting $j$ into $A_i$ shows $A_i = A_j$, and the conclusion is trivial.
	
	Suppose instead $i \ne j$.
	
	Let $\mathcal{J}_\textnormal{donor}$ be the sub-multiset of all paths (and their multiplicities) in $\mathcal{J}$ whose $j$th components pass through $A_j$ but whose $i$th ones do not pass through $A_i$ and $\mathcal{J}_\textnormal{recipient}$ the sub-multiset of all paths whose $i$th components pass through $A_i$ but whose $j$th ones do not pass through $A_j$.  Then \res{cnd:subsume} is equivalent to the condition that all paths of $\mathcal{J}'$ whose $i$th component passes through $A_i$ are alive with respect to $\times_{n \in \IntsPos \setsub \{i,j\}} \{1, \cdots, p_n\} \times A_i \times A_j$.  As for $\mathcal{J}$, each path in $\mathcal{J}_\textnormal{donor}$ is a potential donor to each path in $\mathcal{J}_\textnormal{recipient}$ (with the donations occurring in the $j$th components).  Therefore, by \res{pro:projection}, it suffices to show \begin{enumeratecnd*} \item \label{cnd:donor}
	$\card{\mathcal{J}_\textnormal{donor}} \geq \card{\mathcal{J}_\textnormal{recipient}}$.
	\end{enumeratecnd*}
To invoke 
	\res{cnd:paths}, observe
	\begin{align*}
	\textnormal{card}\left({\proj_i(\mathcal{J}) \cap A_i}\right) &=  \card{\mathcal{J}_\textnormal{recipient}} + \card{ \mathcal{J} \cap \times_{n \in \IntsPos \setsub \{i,j\}} \{1, \cdots, p_n\} \times A_i \times A_j}
	\\
	\textnormal{card}\left({\proj_j(\mathcal{J}) \cap A_j}\right) &=  \card{\mathcal{J}_\textnormal{donor}} + \card{ \mathcal{J} \cap \times_{n \in \IntsPos \setsub \{i,j\}} \{1, \cdots, p_n\} \times A_i \times A_j}
	,\end{align*}
	which by plugging into \res{cnd:paths} gives \ref{cnd:donor}.
\end{proof}

\remStar If $j = 1$, the proposition is trivial.

\propForCor{redirect}

\begin{corollary}
	\label{cor:redirect:proportions}
	Fix arbitrary $N \in \IntsPos$, $i, j \in \inds$, and $\mathcal{J} \subseteq \times_{n \in \inds} \{p_1, \cdots, p_N\}$ such that, for $n \in \{i,j\}$,  
	\[
	\proj_n(\mathcal{J}) = \left\{1^{N!_p/p_n}, \cdots, p_n^{N!_p/p_n}\right\}
	\numberPre{Ct}\label{ctr:two}
	.\]  Let $A_i$ and $B_i$ be ordinary sets that partition $\{1, \cdots, p_i\}$ and, similarly, $\{A_j, B_j\}$ a partition of $\{1, \cdots, p_j\}$; and suppose $\frac{\card{A_i}}{\card{B_i}} \leq \frac{\card{A_j}}{\card{B_j}}$. Then there exists $\mathcal{J}'$ such that (i) $\proj_n(\mathcal{J}') = \proj_n(\mathcal{J}) $ for all $n \in \inds$ and (ii) for all $v \in \mathcal{J}'$, $v_i \in A_i \implies v_j \in A_j$.
\end{corollary}
\begin{proof}
	Let $a_i := \card{A_i}$.
	\begin{align*}
	\left( \frac{\card{A_i}}{\card{B_i}} \leq \frac{\card{A_j}}{\card{B_i}} \right)
	&\iff 
	\left( \frac{a_i}{\card{\{1, \cdots, p_i\} \setsub A_i}} \leq \frac{\card{A_j}}{\card{\{1, \cdots, p_j\} \setsub A_j}} \right) 
	\\&\iff 
	\left( \frac{a_i}{p_i - a_i} \leq \frac{\card{A_j}}{p_j - \card{A_j}} \right)
	\\&\iff 
	\left( a_i p_j - a_i \card{A_j} \leq p_i \card{A_j} - a_i \card{A_j} \right) 
	\\&\iff
	\left(p_j a_i \leq p_i \card{A_j}\right)
	.\end{align*}  By \res{ctr:two}, $\card{A_j} N!_p/p_j$ paths pass through, in their $j$th components, $A_j$, whereas only $a_i N!_p/p_i \leq \frac{p_i}{p_j}\card{A_j} \frac{N!_p}{p_i} = \card{A_j} N!_p/p_j$ paths pass through, in their $i$th components, $A_i$.  Thus \res{pro:redirect} can be invoked.
\end{proof}

The upshot of \res{cor:redirect:proportions} is the following lemma.

\begin{lem}[Intersecting Products]
	\label{lem:redirect:one}
	Fix arbitrary $N \in \IntsPos$ and partitions $(A_n, B_n)$ of $\{1, \cdots, p_n\}$ for $n = 1, \cdots, N$.  Let $n_* \in \argmin_{n \in \inds} \frac{\card{A_n}}{\card{B_n}}$.  Then there exists $\mathcal{J}$ satisfying \res{cnd:multiset} for which 
	\[
	\card{\mathcal{J} \cap \times_{n \in \inds} B_n} = \frac{N!_p}{p_{n_*}} \card{B_{n_*}}
	.\]
\end{lem}
\begin{proof}
	Start with any $\mathcal{J}$ satisfying \res{cnd:multiset}, {\eg} $\times\limitsOne \{1, \cdots, p_i\}$.  By construction of $n_*$, \res{cor:redirect:proportions} can be iteratively applied over all $n \in \inds$ path components, with $i := n$ and $j := n_*$.
		
	Then, in the final $\mathcal{J}'$ for all $n \in \inds$, there are $\frac{N!_p}{p_{n_*}} \card{B_{n_*}}$ paths passing through, in their $n$th components, $B_n$ and, in the $n_*$th ones, $B_{n_*}$.  Moreover, every such path cannot pass through, in its $n'$th component, some $A_{n'}$ by construction of $\mathcal{J}'$ (doing so would have to pass through, in the $n_*$th component, $A_{n_*}$, a contradiction).  Therefore, all of these paths are in $\times_{n \in \inds} B_n$.
\end{proof}

\newcommand{\indRan}{\left\{i^1, \cdots, i^m\right\}}

The following two corollaries give simpler approximations of \res{pbm:extreme:PR} than \res{thm:range:prime}, and are collectively adequate, {\ie} without further use of the theorem, for ascertaining the probability of {\Primes} and residue classes.

\begin{corollary}[Product]
	\label{cor:lub:lower:product}
	Fix an arbitrary $S \subseteq \Ints$.  Suppose, for all $N \in \IntsPos$ and $\mathcal{I}_n \subseteq \Ints$ (for all $n \in \IntsPos$),
	\[s^{-1}(S \bmod N!_p) \supseteq 
	\times_{n \in \inds} \mathcal{I}_n
	\numberPre{A}\label{apx:shifts:product}
	.\]
	
	Then 
	\[\ext_{\mu \in PR} \mu(\probset) \rel \inf_{n \in \IntsPos} \frac{\card{\mathcal{I}_n}}{p_n}
	.\]
\end{corollary}
\begin{proof}
		Combining \res{apx:shifts:product} and \res{apx:optimization:primes} yields \[\ext_{\mu \in PR} \mu(\probset) \rel 
	\lim_{N \to \infty}
	\max_{\mathcal{J}}\frac{
		\card{
			\mathcal{J}
			\cap \times_{n \in \inds} \mathcal{I}_n
	}}
	{N!_p}
	,\]
	which, by \res{lem:redirect:one}, is lower bounded by
	\[
	\lim_{N \to \infty}
	\frac{
		\min_{n \in \inds} \frac{N!_p}{p_n}\card{\mathcal{I}_n}	 
	 }
	{N!_p}
	\geq
	\inf_{n \in \IntsPos} \frac{\card{\mathcal{I}_n}}{p_n}
	.\]
\end{proof}

In the next corollary, the total number of paths that can be made to pass through $\mathcal{I}_n$ in their $n$th components is once again the basis for bounding \res{pbm:extreme:PR}, albeit this time such a bound cannot be further simplified.  That is because $s^{-1}(S \bmod N!_p)$ may no longer be well approximated by a product as it was in \res{apx:shifts:product}; instead it is approximated by a product $\times_{n \in \inds} \mathcal{I}_n$ minus another one, say $\times_{n \in \inds} \mathcal{K}_n$.  Because of this complication, paths must be re-arranged more intricately to realize the maximum in \res{apx:optimization:primes}.  The ability to re-arrange enough paths to ``make up for" the subtraction of $\times_{n \in \inds} \mathcal{K}_n$ depends on the relative sizes of the sets $\mathcal{I}_n$ and $\mathcal{K}_n$.  (Namely, \res{apx:thin:count} below suffices.)

\newcommand{\nInK}{m}
\begin{corollary}[Difference of Products]
\label{cor:product:union}
Fix an arbitrary $S \subseteq \Ints$ and let
for all $n \in \inds$ (recall \res{def:coordinates}), 
\begin{enumerate} 
	\item $\mathcal{I}_n, \mathcal{K}_n \subseteq \Ints$,
	\item $\mathcal{H}_n := \mathcal{I}_n \setsub \mathcal{K}_n$, \item $I_n := \card{\mathcal{I}_n}$,
	\item $K_n := \card{\mathcal{K}_n}$, and
	\item $H_n := \card{\mathcal{H}_n}$
	.\end{enumerate} 
	Further, let $\inds_* := \{k \in \inds : \mathcal{H}_k \ne \emptyset\}$ and $N_* := \card{\inds_*}$; and, for all $N \in \IntsPos$, $n \in \inds$, $m \in \IntsPos$ and $\{i_1, \cdots, i_m\} \subset \IntsPos$, let $\inds_*^{-i} := \inds_* \setsub \{i_1, \cdots, i_m\}$ and $\inds^{-i} := \inds \setsub \left(\inds_* \cup \{i_1, \cdots, i_m\}\right)$.  Note $\left\{\{i_1, \cdots, i_m\},  \inds_*^{-i}, \inds^{-i}\right\}$  partitions $\mathcal{N}$. 

Suppose, for all $N \in \IntsPos$,
\[s^{-1}(S \bmod N!_p) \supseteq \mathcal{I}^N
\numberPre{A}\label{apx:thin}
,\]
in which $\mathcal{I}^N := \times_{n \in \inds} \mathcal{I}_n
\setsub \times_{n \in \inds} \mathcal{K}_n$ and
\newcommand{\nOrgans}{m}
\begin{align*}
&\sum_{\nInK = 0}^{N_* - 2}
(N_*  - 1 - \nInK) \sum_{\begin{smallmatrix}\{i_1, \cdots, i_m\} \subseteq \inds_* :\\ i_1 < \cdots < i_m
\end{smallmatrix}}
\left[
\prod_{n \in \{i_1, \cdots, i_m\}} K_n
\prod_{n \in \inds_*^{-i} } H_n \prod_{n \in \inds^{-i}} I_n
\right] \geq
\prod_{n \in \inds} K_n
\numberPre{A}\label{apx:thin:count}
.\end{align*}
	
	Then 
	\[\ext_{\mu \in PR} \mu(\probset) \rel
	\lim_{N \to \infty}
		\frac{
		\prod_{n \in \inds} I_n}
	{N!_p}
.\]
\end{corollary}
\begin{proof}
	
	Combining Approximations \ref{apx:thin} and \ref{apx:optimization:primes} yields \[\ext_{\mu \in PR} \mu(\probset) \rel 
	\lim_{N \to \infty}
	\max_{\mathcal{J}}\frac{
		\card{
			\mathcal{J}
			\cap \mathcal{I}^N
		}}
	{N!_p}
	\numberPre{A}\label{apx:}
	.\]
	
	By definition of $s^{-1}$ (\ref{def:primes}), for all $n \in \inds$, $\mathcal{I}_n$ and $\mathcal{K}_n$ are subsets of $\{1, \cdots, p_n\}$.  Hence,	$\times_{n \in \inds} \{1, \cdots, p_n\}$ contains $ \cup_{m \in \{0,\cdots, N_* - 2\}, i_1<\cdots<i_m} \times_{n \in \{i_1, \cdots, i_m\}}\mathcal{K}_n \times \times_{n \in \inds_*^{-i}} \mathcal{H}_n \times \times_{n \in \inds^{-i}} \mathcal{I}_n$.

	Every $k \in \times_{n \in \inds} \mathcal{K}_n$ can be enlivened by exchanging its $n$th component with the same of any element $k_* \in \times_{n \in \{i_1, \cdots, i_m\}}\mathcal{K}_n \times \times_{n \in \inds_*^{-i}} \mathcal{H}_n \times \times_{n \in \inds^{-i}} \mathcal{I}_n$ for any $m \in \{0,\cdots, N_* - 2\}$ (in which $\{0, \cdots, \emph{negative number}\} := \{0\}$), $i_1<\cdots<i_m$, and $n \in \inds_*^{-i}$.  For example, if $N = 2$, $\mathcal{I}_1 = 1 = \mathcal{K}_1 = \mathcal{K}_2$, $\mathcal{I}_2 = \{1,2\}$, then $\mathcal{H}_1 = \emptyset$, $\mathcal{H}_2 = \{2\}$, $\mathcal{N}_* = \{2\}$, $N_* - 2 < 0$, and $k = (1, 1)$ is (dead and) can be enlivened by exchanging its second component with the same of $(1, 2) \in \mathcal{I}_1 \times \mathcal{H}_2$.  In this simple case with no overlapping paths, the live and dead paths can be distinguished by thickness (or, alternatively, color) with thick denoting alive:
	 \begin{tikzpicture}[
		  font=\sffamily,
		  every matrix/.style={ampersand replacement=\&,column sep=.5cm,row sep=.1cm},
		  source/.style={draw,thick,rounded corners,fill=yellow!20,inner sep=.3cm},
		  process/.style={draw,thick,circle,fill=yellow!20},
		  to/.style={->,shorten >=1pt,semithick,font=\sffamily\footnotesize},
		  every node/.style={align=center,scale=.8}]
		
\matrix{
		      			\& \node[process] (daq) {2}; 
		      			\\
		  \node[process] (1) {1};\& \node[process] (buffer) {1};
		    \\
		  };
		
\draw[line width = .5mm] (1) -- node[midway,right] {} (daq);
		  \draw[draw=black] (1) -- node[midway,right] {} (buffer);
		\end{tikzpicture}.  
	
	Returning to the general case, every such $k_*$ could make $ N_* - 1 -m $ donations before another exchange would kill it or there were no more recipients in $\times_{n \in \inds} \mathcal{K}_n$, which would only be the case if every path therein had been enlivened.  When inequality \ref{apx:thin:count} is satisfied, the latter is possible.  Beginning with $\times_{n \in \inds} \{1, \cdots, p_n\}$ and exhausting donations yields some $\mathcal{J}$ satisfying \res{cnd:multiset}.  Because the foregoing procedure begins with $\card{I^N}$ lives and performs at least $\prod_{n \in \inds} K_n$ donations, $\mathcal{J}$ satisfies
	\[\card{
			\mathcal{J}
			\cap \mathcal{I}^N
	}
	\geq
	\prod_{n \in \inds}\card{\mathcal{I}_n}
	\numberPre{A}\label{apx:donated}   
	.\]  Combining Approximations \ref{apx:donated} and \ref{apx:} concludes.
\end{proof}

\newcommand{\ctrset}{X}

\section{Application to {\Primes}}
\label{sec:application-to-primes}

For an application of \res{thm:range:prime} or one of its corollaries (\ref{cor:lub:lower:product} and \ref{cor:product:union}) to a given $S$, the approximation of $s^{-1}(S \bmod N!_p)$ for all $N$ is fundamental.  The following propositions concern the intersections of residue classes.  When there is an infinite number of intersections, there is a simple characterization:
\begin{pro}[Infinite Intersections]
	\label{pro:resinf}
	Suppose \begin{enumerate*}\item $\mathcal{I}$ is a set, \item for all $i \in \mathcal{I}$, $j_i,m_i \in \IntsPos$, and \item \label{cnd:unbounded} $\sup_{i \in \mathcal{I}} m_i = \infty$.
		\end{enumerate*}
	Then 
	\[
	\cap_{i \in \mathcal{I}} j_i \bmod m_i
	\numberthis\label{exr:intersect:infinite}
	\] can have at most one element $s$, in which case, for all $i \in \mathcal{I}$, \[j_i \bmod m_i = s \bmod m_i
	.\]
\end{pro}
\begin{proof}
	The three suppositions guarantee Set \ref{exr:intersect:infinite} is well defined.  The gaps between elements of (\ref{exr:intersect:infinite}) must be at least $m_i - 1$ for all $i \in \mathcal{I}$, implying the gaps must be arbitrarily close to $\sup_{i \in \mathcal{I}} m_i - 1 = \infty$, by Supposition \ref{cnd:unbounded}.  This excludes the possibility of multiple elements.  One element, say $s$, is possible precisely when, for all $i \in \mathcal{I}$,  $s \in j_i \bmod m_i$.  But for every given $i$, $s \in j_i \bmod m_i$ iff $j_i \bmod m_i = s \bmod m_i$.
\end{proof}

The following proposition examines the nature of individual residue classes within a non-empty intersection thereof.
\begin{pro}[Computation of Shift]
	\label{pro:shift}
	Suppose \begin{itemize}
		\item $N \in \IntsPos$,
		\item $m_1, \cdots, m_N$ are co-prime,
		\item $s_1, \cdots, s_N$ are positive integers,
		\item $k \in \{1, \cdots, N\}$, and
		\item $m_k$ divides the shifts of $\cap_{i=1}^N s_i \bmod m_i$
	.\end{itemize}  Then $s_k \in 0 \bmod m_k$.
\end{pro}
\begin{proof}
	By the Chinese Remainder Theorem, $\cap_{i=1}^N s_i \bmod m_i$ is indeed a residue class.
	
	Suppose the shifts of $\cap_{i=1}^N s_i \bmod m_i$ are divisible by $m_k$.  Then 
	\[(0 \bmod m_k) \cap \left(\cap_{i=1}^N s_i \bmod m_i\right) \ne \emptyset.\]  In particular, $(0 \bmod m_k) \cap (s_k \bmod m_k)$ is non-empty.  Hence, $s_k \in 0 \bmod m_k$.
\end{proof}

\propForCor{shift}

\begin{corollary}
	\label{cor:shiftmodco}
	Suppose $s_1, \cdots, s_N$ are integers and for all $i \in \{1, \cdots, N\}$ $s_i \notin 0 \bmod p_i$. Then the shift and modulus of $\cap_{i=1}^N s_i \bmod p_i$ are co-prime.
\end{corollary}
\begin{proof}
	Suppose, contrarily, that the shift and modulus of $\cap_{i=1}^N s_i \bmod p_i$ have a common divisor greater than 1, it is equal to $p_k$ for some $k \in \{1, \cdots, N\}$.  Then the previous proposition gives $s_k \in 0 \bmod p_k$.
\end{proof}

Finally a computation of $s^{-1}$:

\begin{lem}[Primes in the Intersection of Residue Classes]
	\label{lem:intersect:nonempty}
	For all $N \in \IntsPos$, \[
	s^{-1}\left(\mathbb{N}_p \bmod N!_p\right)
	\supseteq
	\times_{n \in \inds}
	\{
	1, \cdots, p_n-1
	\} 	
	.\]
\end{lem}
\begin{proof}
	If, for all $n \in \inds$, $j_n \in \{1, \cdots, p_n - 1\}$, then, by \res{cor:shiftmodco}, the shift and modulus of $\cap_{i=1}^N j_i \bmod p_i $ are co-prime.  Hence, by Dirichlet's theorem on the distribution of primes in residue classes, $\cap_{i=1}^N j_i \bmod p_i  \cap \mathbb{N}_p \ne \emptyset$.
\end{proof}

\begin{thm}[Probability Range of Primes]
$\{\mu(\mathbb{N}_p) : \mu \in PR\} = [0, 1/2]$.

\label{thm:primesRan}
\end{thm}
\begin{proof}
	 By \res{cor:lub:lower:product}, with \res{apx:shifts:product} given by \res{lem:intersect:nonempty} and $\mathcal{I}_n := \{1, \cdots, p_n - 1\}$,
	\[
	\sup_{\mu \in PR} \mu(S) \geq \inf_{i \in \IntsPos} \frac{\card{\mathcal{I}_i}}{p_i} = \inf_{i \in \IntsPos} \frac{p_i-1}{p_i} = \frac{1}{2}
	,\]
	with $i = 1$ attaining the infimum.

$\sup_{\mu \in PR} \mu(S) \leq 1/2$ by virtue of the primes inclusion in the odd numbers union $\{2\}$.  Therefore the least upper bound is $1/2$.  The greatest lower bound of $0$ is the content of \res{pro:lower:primes}.
\end{proof}

\section{Application to Residue Classes}\label{sec:app:res}

Because the primes have zero mass under $R$ \citep{kad}, we know a priori that there must be a residue class on which $PR$ and $R$ can disagree. It turns out the residue classes that have non-singleton bounds under $PR$ are those mod neither 1 nor a prime; that is all residue classes on which $PR$ was not initially defined!

The following two lemmas instantiate Approximations \ref{apx:shifts:product} and \ref{apx:thin}, respectively.

\newcommand{\indS}{\mathcal{I}^{S}_n}

\begin{lem}[Intersection of Residue Classes]
	\label{lem:intersectNonempty:res}
	 Let $r,m \in \IntsPos$ and $S := r \bmod m$ be an arbitrary residue class.  For all $n \in \IntsPos$, let $\indS := \{1 : p_n\}$ if $p_n \ndiv m$ and $r \bmod p_n \cap \{1 : p_n\}$ otherwise.  For all $N \in \IntsPos$,
	\[
	s^{-1}(S) \supseteq \times_{n \in \inds} \indS
	\]
\end{lem}
\begin{proof}
	For all $j \in \times_{n \in \inds} \indS$, 
	\begin{align*}
	\cap_{n \in \inds} j_n \bmod p_n \cap S 
	&= \cap_{n \in \inds: p_n | m}  j_n  \bmod p_n \cap \cap_{n \in \inds:\neg p_n | m} j_n \bmod p_n \cap S \\&= \cap_{n \in \inds: p_n | m} r  \bmod p_n \cap \cap_{n \in \inds:\neg p_n | m}  j_n \bmod p_n \cap S \\&= \cap_{n \in \inds:\neg p_n | m} j_n \bmod p_n \cap S
	\numberPre{}\label{exr:nondivisor}
	,\end{align*} the last equality holds because $r \bmod p_n \subseteq S$ for all $p_n$ dividing $m$.  By the Chinese Remainder Theorem, set \ref{exr:nondivisor} is non-empty.
\end{proof}

\begin{lem}[Intersecting a Residue Class Union]
	\label{lem:intersectNonempty:union}
	Suppose $r,m \in \IntsPos$ is such that $m$ is made up of two (not necessarily unique) prime factors $m_1, m_2$.  Let $S := \Ints \setsub r \bmod m$.  For all $N \in \IntsPos$,
		\[
		s^{-1}\left( S \bmod N!_p \right)
		\supseteq 
		\left[\times_{n \in \inds} \{1 : p_n\}\right] 
		\setsub
		\left[
		(\times_{n \in \inds \setsub \{i_1, i_2\}} \{1 : p_n\}) \times \times_{n \in \{i_2, i_2\}} r \bmod p_n
		\right],
		\numberPre{A}\label{apx:shifts:union}
		\]
	in which $i_1$ and $i_2$ are such that $p_{i_1} = m_1$ and $p_{i_2} = m_2$.
	
	If, further, $i_1 = i_2$,
	\[
	s^{-1}\left( S \bmod N!_p \right)
	\supseteq 
	\times_{n \in \inds} \{1 : p_n\},
	\numberPre{A}\label{apx:shifts:one}
	\]
\end{lem}
\begin{proof}
	Fix an element $j$ in the {\RHS} 
	\[
	\left[\times_{n \in \inds} \{1 : p_n\}\right] \setsub
	\left[
	(\times_{n \in \inds \setsub \{i_1, i_2\}} \{1 : p_n\}) \times \times_{n \in \{i_2, i_2\}} r \bmod p_n
	\right]
	\numberPre{}\label{exr:product:sub}
	\] of \res{apx:shifts:union}.
	Then, if $i_1 \ne i_2$, for some $i_* \in \{i_1, i_2\}$, \[S_* := j_{i_*} \bmod p_{i_*} \subset S
	\numberPre{}\label{exr:divide:inclusion}
	.\]
	In that case,
	\begin{align*}
	\cap_{n \in \inds} j_n \bmod p_n \cap S 
	&= \cap_{n \in \{i_1, i_2\}}j_n \bmod p_n \cap \cap \cap_{n \in \inds \setsub \{i_1,i_2\}} j_n \bmod p_n \cap S
	\\&=
	(S_* \cap S) \cap
	\cap_{n \in \{i_1, i_2\}\setsub\{i_*\}}j_n \bmod p_n \cap \cap_{n \in \inds \setsub \{i_1,i_2\}} j_n \bmod p_n
	\\&=
	 \cap_{n \in \{i_1, i_2\}}j_n \bmod p_n \cap \cap_{n \in \inds \setsub \{i_1,i_2\}} j_n \bmod p_n
,\end{align*} the latter equality by \res{exr:divide:inclusion}.  In summary,
	\[
	\cap_{n \in \inds} j_n \bmod p_n \subset S
	\numberPre{A}\label{apx:absorb}
	\]
	By the Chinese Remainder Theorem, the {\LHS} of \res{apx:absorb} is non-empty.  In particular, $S \bmod N!_p \cap \cap_{n \in \inds} j_n \bmod p_n \ne \emptyset$. Because $j$ was an arbitrary element of set \ref{exr:product:sub}, \res{apx:shifts:union} holds.
	
	If, instead, $r = m_1^2$ and $j$ in the {\RHS} of \res{apx:shifts:one}, then \res{apx:absorb} still holds.
\end{proof}

\begin{thm}[Probability Range of Residue Class]
\label{thm:resRan}
For every $r \in \Ints$ and $m \in \IntsPos \setsub$ $\{p_0, p_1, \cdots\}$,
\[\{\mu\left(r \bmod m\right): \mu \in PR\} = \left[0, \frac{1}{\max\{p \in \mathbb{N}_p: p|m\}}\right]
.\]
\end{thm}
\begin{proof}
	With \res{apx:shifts:product} given by \res{lem:intersectNonempty:res} and $\mathcal{I}_n := \mathcal{I}_n^S$, by \res{cor:lub:lower:product},
	\[
	\sup_{\mu \in PR} \mu(S) \geq \inf_{n \in \IntsPos} \frac{\card{\mathcal{I}_n^S}}{p_n} = \inf_{n \in \IntsPos: p_n | m} \frac{1}{p_n} = \frac{1}{\max\{p \in \mathbb{N}_p: p|m\}}
	.\]
	
	For $p_n$ dividing $m$, $\mu(S)$ can be no larger than the measure $\mu\left(r \bmod p_n\right)$ of a set containing it, so 
	\[
	\sup_{\mu \in PR} \mu(S) = \frac{1}{\max\{p \in \mathbb{N}_p: p|m\}}
	.\]

	It remains to show merely 
	\[\inf_{\mu \in PR} \mu(S) = 0
	\numberPre{Id}\label{idn:zero}.\]
	
	It suffices to consider $m$ with at most two prime factors, as every residue class of a modulus with more prime factors is a subset of some class with two; and thus has probability at most that of the superset.  Classes with one prime factor are defined.  Therefore, it suffices to consider moduli with exactly two prime factors $m_1$ and $m_2$.
	
	Since our theory is based on suprema rather than infima, consider the set complement $S^c$, itself a union of residue classes.  Claimed \res{idn:zero} is equivalent to $\sup_{\mu \in PR} \mu(S^c) = 1$, which in turn is equivalent to \[
	\sup_{\mu \in PR} \mu(S^c) \geq 1 \numberPre{A}\label{apx:union}\]
	by definition of finitely additive probability.

	If $m_1 = m_2$, then by \res{cor:lub:lower:product},
	with \res{apx:shifts:product} given by \res{apx:shifts:one} and $\mathcal{I}_n := \{1, \cdots, p_n\}$, 
\[
\sup_{\mu \in PR} \mu(S^c) \geq \inf_{n \in \IntsPos} \frac{\card{\{1, \cdots,p_n\}}}{p_n} = \inf_{n \in \IntsPos} \frac{p_n}{p_n} = 1
.\]	

	Now consider $m_1 \ne m_2$.  Fix $N \in \IntsPos$.  Let, for all $n \in \inds$, $\mathcal{I}_n := \{1 : p_n\}$ and \[\mathcal{K}_n :=
	\begin{cases}  
	 \{1 : p_n\} & n \in \inds \setsub \{i_1, i_2\} 
	 \\
	 r \bmod p_n \cap \{1 : p_n\} & n \in \{i_2, i_2\}
	 \end{cases},\]
	  in which $i_1$ and $i_2$ are such that $p_{i_1} = m_1$ and $p_{i_2} = m_2$.  (Compare to the setup of \res{lem:intersectNonempty:union}.)
	  Then
	  \[
	  \card{\mathcal{K}_n} = \begin{cases}
	  \card{\mathcal{I}_n} & n \in \inds \setsub \{i_1, i_2\}
	  \\
	  1 				   & n \in \{i_1, i_2\} 
	  \end{cases}
	  \numberPre{Id}\label{idn:toxic}
	  \]
	  and $\inds_* := \{n \in \inds : \mathcal{I}_n \setsub \mathcal{K}_n \ne \emptyset\} = \{i_1, i_2\}$.  Plugging $N_* = \card{\{i_1, i_2\}} = 2$ into the {\LHS} of \ref{apx:thin:count} yields
	\begin{align*}
	&\sum_{\nInK = 0}^{0}
	(2  - 1 - \nInK) \sum_{\{i_1, \cdots, i_m\} \subseteq \inds_* : i_1 < \cdots < i_m} 
	\left[
	\prod_{n \in \emptyset} K_n
	\prod_{n \in \inds_* } \card{\mathcal{I}_n \setsub \mathcal{K}_n} \prod_{n \in \inds \setsub \inds_*} \card{\mathcal{I}_n}
	\right]
	\\
	=&
	\prod_{n \in \inds_* } \card{\mathcal{I}_n \setsub \mathcal{K}_n} \prod_{n \in \inds \setsub \inds_*} \card{\mathcal{I}_n}
	,\end{align*}
	matching \res{idn:toxic} and thereby satisfying inequality \ref{apx:thin:count}.
	
	Therefore, with \res{apx:thin} given by (i) \res{apx:shifts:union} of \res{lem:intersectNonempty:union} and (ii) $S^c$ in place of $S$, by \res{cor:product:union},
	\begin{align*}
	\ext_{\mu \in PR} \mu(\probset) &\rel
	\lim_{N \to \infty}
	\frac{
		\prod_{n \in \inds}\card{\mathcal{I}_n}}
	{N!_p}
	\\&=
	\lim_{N \to \infty}
	\frac{
		\prod_{n \in \inds} p_n}
	{N!_p}	
	\\&= 1 & \textnormal{\res{def:primes}}
	,\end{align*} 
	which, in turn, gives \res{apx:union}, all that remained to prove.
\end{proof}

\remStar The proof of the lower bound in the case $m = p^2$ for some prime $p$ could have appealed to \res{cor:lub:lower:product} rather than \ref{cor:product:union}, but the former has more conditions to check.

\section{Conclusion}
\label{sec:conclusion}

$PR$ is a distinct family of uniform finitely additive probabilities over $\Ints$.  From the strict inclusions
\begin{equation*}
WT \subset L \subset S \subset R \subset PR,
\end{equation*}
we have shown the last.  We have given necessary and sufficient conditions for there to exist a probability charge in $PR$ assigning $w$ to $\Primes$, namely $w \in [0, 1/2]$.  We have also given necessary and sufficient conditions for the existence of a probability charge in $PR$ assigning $x$ to a class modulo $c$, where $c$ is composite, namely $x \in [0, 1/y]$, where $y$ is the largest prime factor of $c$.


\begin{thebibliography}{1}
\providecommand{\url}[1]{{#1}}
\providecommand{\urlprefix}{URL }
\expandafter\ifx\csname urlstyle\endcsname\relax
  \providecommand{\doi}[1]{DOI \discretionary{}{}{}#1}\else
  \providecommand{\doi}{DOI \discretionary{}{}{}\begingroup
  \urlstyle{rm}\Url}\fi

\bibitem{kad}
J.B. Kadane, A.~O'Hagan, Using finitely additive probability: Uniform
  distributions on the natural numbers, Journal of the American Statistical
  Association \textbf{90}, 626 (1995)

\bibitem{sch}
O.~Schirokauer, J.B. Kadane, Uniform distributions on the natural numbers,
  Journal of Theoretical Probability \textbf{20}(3), 429 (2007)

\bibitem{van}
E.K. van Douwen, Finitely additive measures on $\mathbb{N}$, Topology and its
  Applications \textbf{47}(3), 223 (1992)

\bibitem{ker}
T.~Kerkvliet, R.~Meester, Uniquely determined uniform probability on the
  natural numbers, Journal of Theoretical Probability \textbf{29}(3), 797
  (2016)

\end{thebibliography}
\end{document}